\documentclass[11pt, oneside]{article}   	
\usepackage{geometry}                		
\usepackage[parfill]{parskip}    		
\usepackage{graphicx}				
\usepackage{amssymb}
\usepackage{amsmath}
\usepackage{amsthm}
\usepackage{multirow}
\usepackage[figure,noline,linesnumbered]{algorithm2e}
\SetAlFnt{\small}

\theoremstyle{plain}
\newtheorem{theorem}{Theorem}
\newtheorem{lemma}{Lemma}
\newtheorem{corollary}{Corollary}

\makeatletter
\def\thm@space@setup{%
  \thm@preskip=\parskip \thm@postskip=0pt
}
\makeatother

\title{Targeted Fibonacci Exponentiation}
\author{Burton S. Kaliski Jr.\footnote{{\tt bkaliski@alum.mit.edu}.  The views expressed are my own and do not necessarily reflect those of my employer.}}
\date{Version 1.0.1 --- November 4, 2017}							

\begin{document}
\maketitle

\begin{abstract}
A targeted exponentiation algorithm computes a group exponentiation operation $a^k$ with a reversible circuit in such a way that the initial state of the circuit consists of only the base $a$ and fixed values, and the final state consists of only the exponential $a^k$ and fixed values.  Three targeted exponentiation algorithms based on Fibonacci addition chains are considered, offering tradeoffs in terms of the number of working registers and the number of iterations.  The approaches also motivate related results on the Fibonacci Zeckendorf array, including a new \emph{modular Hofstadter G problem} and an improvement to Anderson's recent algorithm for locating pairs of adjacent integers in the extended Fibonacci Zeckendorf array.  The algorithms have applications in quantum computing.  
\end{abstract}

\section{Introduction}

Let $a$ be an element of a group $H$, and let $k$ be a positive integer.  The \emph{group exponentiation} of $a$ to the power $k$ is the group element $a^k$ (where the group operation is written multiplicatively).  Group exponentiation can be computed by many different algorithms, including numerous approaches based on binary representations as well as some based on Fibonacci representations \cite{byrne2007comparison} \cite{klein2008should} \cite{meloni2007new}.  

Although binary exponentiation algorithms are generally more efficient than Fibonacci algorithms, this advantage is primarily for classical, non-reversible computing, not for reversible computing in general.  In particular, if the base $a$ is variable, Fibonacci exponentiation may be preferable because its basic step --- the mapping $(a, b) \mapsto (b, ab)$ --- is inherently reversible (\cite{perumalla2013introduction}, Sec. 11).\footnote{This is under the technical condition that the elements involved remain invertible, which can be assured with appropriate parameter choices.}  In contrast, a reversible implementation of the basic step in binary exponentiation --- the mapping $a \mapsto a^2$ --- carries forward the input along with the output.  As a consequence, a reversible circuit for binary exponentiation with a variable base $a$ needs registers for each of the successive squares in the binary ``addition chain.''  A reversible circuit for Fibonacci exponentiation, in contrast, only needs registers for the latest values of $a$ and $b$.  The core Fibonacci addition chain evolves in place.

In either case, a basic reversible exponentiation circuit may produce as output not only the exponential $a^k$, but also other ``side values,'' possibly including the input $a$ itself.  For certain applications, especially in quantum computing, the presence of these additional values can be problematic in terms of their effect on subsequent computation.  In these applications, it is preferable for the circuit to perform a \emph{targeted exponentiation} where the initial state of the registers in the reversible circuit consists of only the base $a$ and fixed ``ancilla'' values (or a ``clean ancilla,'' in the terminology of H{\"a}ner \emph{et al.} \cite{haner2016factoring}), and the final state consists of only the exponential $a^k$ and fixed ``garbage'' values.  Such an approach assumes that $a \mapsto a^k$ is an invertible mapping, i.e., that $\textit{kInv} = k^{-1} \bmod{r}$ exists (equivalently, that $k$ and $r$ are relatively prime, where $r$ is the order of $a$ in $H$).

The rest of the paper is organized as follows.  After preliminaries in Section \ref{section-preliminaries}, the paper presents three algorithms for targeted exponentiation.  A basic approach based on a low-to-high Fibonacci exponentiation algorithm is described in Section \ref{section-basic}.  Two improvements follow:  a ``dual'' approach  in Section \ref{section-dual} based on a variant of the low-to-high algorithm, and a hybrid approach in Section \ref{section-hybrid} that combines a high-to-low algorithm, here called \emph{Hofstadter G pair exponentiation}, with the low-to-high algorithm.  Quantum computing applications are discussed in Section \ref{section-quantum}.  The paper concludes with suggestions for further research.

Appendices include a proof of correctness of the Hofstadter G pair exponentiation algorithm; a solution to a new \emph{modular Hofstadter G problem}; and an improvement to Anderson's recent algorithm \cite{anderson2014extended} for locating pairs of adjacent integers in the extended Fibonacci Zeckendorf array.

\section{Preliminaries}
\label{section-preliminaries}

\emph{Fibonacci numbers.}  Let $F_0 = 0$, $F_1 = 1$, $F_i = F_{i-1} + F_{i-2}$ for $i \ge 2$ denote the Fibonacci numbers.  Let $\phi = (1 + \sqrt{5}) / 2$ denote the Golden Ratio.  

\emph{Bit strings.}  A bit string $\vec{\nu}$ is a sequence of bits $\langle\nu_1 \ldots \nu_{\|\vec{\nu}\|}\rangle$, where $\nu_i \in \{0,1\}$ and ${\|\vec{\nu}\|}$ denotes the length of $\vec{\nu}$.  Let $\langle \rangle$ denote the empty string.  The concatenation of two bit strings $\vec{\upsilon}$ and $\vec{\beta}$ is written $\vec{\upsilon} \| \vec{\beta}$.

The ``up shift'' of a bit string\footnote{``Up'' and ``down'' are preferred here to the usual ``left'' and ``right,'' to focus on the effect of shifting on the significance of the bits, rather than their position, given that Fibonacci representations often have their least significant bits on the left, whereas binary representations often have theirs on the right.}, denoted $\Uparrow$, moves bits to higher-indexed positions, adding in $i$ new $0$ bits at the lowest-indexed positions:
\begin{equation*}
\vec{\nu} \Uparrow i \stackrel{\Delta}{=} \langle \overbrace{0 \ldots 0}^i \enspace \nu_1 \ldots \nu_{\|\vec{\nu}\|} \rangle, \quad i \ge 0 \quad . 
\end{equation*}
The ``down shift,'' denoted $\Downarrow$ similarly moves bits to lower-indexed positions, dropping off the $i$ lowest-indexed bits:
\begin{equation*}
\vec{\nu} \Downarrow i \stackrel{\Delta}{=} \langle \nu_{i+1} \ldots \nu_{\|\vec{\nu}\|} \rangle, \quad 0 \le i \le \|\vec{\nu}\| - 1 \quad .
\end{equation*}
If $i \ge \|\vec{\nu}\|$, then $\vec{\nu} \Downarrow i = \langle \rangle$.

\emph{Fibonacci sum.}  Let $n$ be a non-negative integer.  The \emph{Fibonacci sum} corresponding to a bit string $\vec{\nu}$, denoted $\textsc{FibSum}(\vec{\nu})$, is defined as
\begin{equation*}
\textsc{FibSum}(\vec{\nu}) \stackrel{\Delta}{=} \sum_{i = 1}^{\|\vec{\nu}\|} \nu_i F_{i+1} \quad .
\end{equation*}
(Here, as is conventional, the sum starts at the second $1$, i.e., $F_2$, rather than at $F_1$.)

\emph{Fibonacci representation.}  If $n = \textsc{FibSum}(\vec{\nu})$, then $\vec{\nu}$ is a \emph{Fibonacci representation} of $n$.  Every positive integer $n$ has at least one Fibonacci representation.  If $n \le F_h - 2$, then there exists at least one Fibonacci representation of $n$ that is no more than $h-3$ bits long.  

\emph{Zeckendorf representation.}  A Fibonacci representation is in \emph{Zeckendorf form} if its most significant bit, i.e.,  $\nu_{\|\vec{\nu}\|}$, is 1, and no two consecutive bits are $1$.  Every positive integer $n$ has exactly one such \emph{Zeckendorf representation} \cite{zeckendorf1972representation} \cite{lekkerkerker1951voorstelling}.  The Zeckendorf representation of an integer may be determined by a ``greedy'' high-to-low algorithm that repeatedly selects the largest possible Fibonacci number less than or equal to the remaining balance in the integer.  If $F_{h-1} \le n \le F_h - 1$, then the Zeckendorf representation of $n$ is $h-2$ bits long.  

\emph{Hofstadter's G sequence} \cite{hofstadter2000godel}, denoted $G(x)$, may be defined recursively as $G(0) = 0$, $G(1) = 1$, and $G(x) = x - G(G(x-1))$ for $x \ge 2$.  The particular form of the recursion is not directly relevant to the applications described here.  However, the following property, shown by Granville and Rasson \cite{granville1988strange} is:
\begin{equation}
\forall x \ge 0, \quad G(x) = \lfloor \phi^{-1} (x + 1) \rfloor \quad .
\end{equation}
A \emph{Hofstadter G pair} is a pair of positive integers $(u, v)$ such that $u = G(v)$.

\emph{Intervals.}  $(x : y)$ denotes the open interval containing all $z$ between $x$ and $y$.

\emph{Modular arithmetic notation.}  Let $r$ be a positive integer.  $(x)_r$ denotes reduction modulo $r$, i.e., $(x)_r \stackrel{\Delta}{=} x \bmod r$, and $(x : y)_r$ denotes the open ``modular interval'' containing all $z$ between $x$ and $y$ modulo $r$.  Specifically:
\begin{itemize}
\item If $(x)_r < (y)_r$, then $(x : y)_r$ contains all $z$ such that $(x)_r < z < (y)_r$.
\item If $(x)_r > (y)_r$, then it contains all $z$ such that either $(x)_r < z < r$ or $0 < z < (y)_r$.
\end{itemize}

\section{Basic approach}
\label{section-basic}

Following general methods in reversible computing, a basic approach to computing a targeted exponentiation $b = a^k$ combines two reversible exponentiation circuits, one for raising to the power $k$ and the other for raising to the power $\textit{kInv} = k^{-1} \pmod{r}$, i.e., the inverse operation $a = b^{\textit{kInv}}$.

The first circuit combines a conventional exponentiation algorithm that maps $a$ and other fixed inputs to $b$ and possibly input-dependent side values with a ``rewinding'' operation that undoes the computation of the input-dependent side values and replaces them with fixed side values.  

The second circuit likewise combines an algorithm that maps $b$ and other fixed inputs to $a$, followed by its own rewinding operation.  

If the two circuits produce the same intermediate value, then if the first is run in the forward direction and the second is then run in reverse, the combination will produce $b$ and other fixed outputs from $a$ and other fixed inputs --- a targeted exponentiation.

Algorithm $\textsc{FibExp}$, shown in Figure \ref{algorithm-FibExp}, may be employed as the core of this approach.  

\begin{algorithm}
\Indp
\SetKwInOut{Input}{Input}
\SetKwInOut{Output}{Output}
\underline{Algorithm $\textsc{FibExp}$} \\
\BlankLine
\Input{A group element $a$ and a non-negative integer $k$}
\Output{$a^k$}
\BlankLine
\Begin{
  Let $\vec{\kappa}$ be a Fibonacci representation of $k$ and let $\ell = \|\vec{\kappa}\|$\;
  $(c, d, e) \leftarrow (1, a, 1)$\;
  \For{$i \leftarrow 1$ \KwTo $\ell$}{
    $(c, d) \leftarrow (d, cd)$\;
    \lIf{$\kappa_i = 1$}{$e \leftarrow de$}
  }
  \KwRet{$e$}\;
}
\caption{Algorithm $\textsc{FibExp}$ computes an exponential by scanning a Fibonacci representation in ``low-to-high'' order.}
\label{algorithm-FibExp}
\end{algorithm}

The proof of correctness of Algorithm $\textsc{FibExp}$ is not detailed here, but may be shown by induction, where the value of $e$ after the $i$th iteration is $a^{k_i}$ for $k_i = \textsc{FibSum}(\vec{\kappa} \Downarrow (\ell - i))$. 

The number of iterations of Algorithm $\textsc{FibExp}$ equals the length $\ell$ of $\vec{\kappa}$.  If $k < r$ and $r < \phi^h$, then $r \le F_{h+2} - 1$, so $k \le F_{h+2} - 2$.  As a result, there exists at least one Fibonacci representation of $\vec{\kappa}$ with length at most $h-1$ bits, and at least one Zeckendorf representation with length at most $h$.  The number of iterations can thus be bounded by at most $h-1$ with an appropriate choice of representation (or by $h$, if Zeckendorf form is chosen).

The bound $k \le F_{h+2} - 2$ also holds for group orders $r$ between $\phi^h$ and $F_{h+2} - 1$.  However, it is convenient for later analysis to bound $r < \phi^h$, or equivalently to set $h \ge \lceil \log_{\phi} r\rceil$.

Algorithm $\textsc{FibExp}$ may be considered a ``low-to-high'' Fibonacci exponentiation algorithm, because bits of the exponent are scanned in increasing order of significance.  The algorithm is thus a counterpart to traditional ``right-to-left'' binary exponentiation algorithm (with the convention that the rightmost bit of a binary representation is the least significant), where in both cases a sequence of powers of the base are selectively multiplied into the result, based on bits of the exponent.

Building a targeted exponentiation algorithm from Algorithm $\textsc{FibExp}$ involves four main steps, following the basic reversible computing approach above:
\begin{enumerate}
\item Run $\textsc{FibExp}(a, k)$ forward.
\item Rewind the Fibonacci chain steps in $\textsc{FibExp}(a, k)$.  This replaces the input-dependent side values with fixed side values.
\item ``Fast forward'' the Fibonacci chain steps in $\textsc{FibExp}(b, \textit{kInv})$.  (This is now the circuit for raising to $\textit{kInv}$, which is being run in reverse, so ''rewind'' becomes ''fast forward.'')
\item Run $\textsc{FibExp}(b, \textit{kInv})$ in reverse.
\end{enumerate}

Algorithm $\textsc{BasicTargetedFibExp}$ (see Figure \ref{algorithm-BasicTargetedFibExp} illustrates the approach.  Each circuit produces the same intermediate values up to a swap operation, $(1, a, b)$.  Not counting temporary values, Algorithm $\textsc{BasicTargetedFibExp}$ requires just three working registers (to save space, the input $a$ can be put into the register that holds $c$; the output $b$ can be taken out of the register that holds $e$).  For purposes of comparison (see Section \ref{section-conclusion}), the algorithm may be considered to have a ``profile'' of $3 \times 4h$:  three registers, roughly $4h$ iterations.

\begin{algorithm}
\Indp
\SetKwInOut{Input}{Input}
\SetKwInOut{Output}{Output}
\SetKw{KwDownTo}{down to}
\underline{Algorithm $\textsc{BasicTargetedFibExp}$} \\
\BlankLine
\Input{A group element $a$ in a group of order $r$, and integer $k$ between $1$ and $r-1$ that is relatively prime to $r$}
\Output{$a^k$}
\BlankLine
\Begin{
  $\textit{kInv} = k^{-1} \pmod{r}$\;
  Let $\vec{\kappa}$ be a Fibonacci representation of $k$ and let $\ell = \|\vec{\kappa}\|$\;
  Let $\vec{\kappa'}$ be a Fibonacci representation of $\textit{kInv}$ and let $\ell' = \|\vec{\kappa'}\|$\;
  \BlankLine
  \tcp{Run $\textsc{FibExp}(a, k)$ forward, producing $(a^{F_{\ell}}, a^{F_{\ell+1}}, a^k)$}
  $(c, d, e) \leftarrow (1, a, 1)$\;
  \For{$i \leftarrow 1$ \KwTo $\ell$}{
    $(c, d) \leftarrow (d, cd)$\;
    \lIf{$\kappa_i = 1$}{$e \leftarrow de$}
  }
  \BlankLine
  \tcp{Rewind Fibonacci chain for $\textsc{FibExp}(a, k)$, producing $(1, a, a^k)$}
  \For{$i \leftarrow \ell$ \KwDownTo $1$}{
    $(c, d) \leftarrow (c^{-1}d, c)$\;
  }
  \BlankLine
  \tcp{Swap second, third registers, producing $(1, a^k, a)$}
  $(c, d, e) \leftarrow (c, e, d)$\;
  \BlankLine
  \tcp{``Fast forward'' Fibonacci chain for $\textsc{FibExp}(b, \textit{kInv})$, producing $(a^{kF_{\ell'}}, a^{kF_{{\ell'}+1}}, a)$}
  \For{$i \leftarrow 1$ \KwTo $\ell'$}{
    $(c, d) \leftarrow (d, cd)$\;
  }
  \BlankLine
  \tcp{Run $\textsc{FibExp}(b, \textit{kInv})$ in reverse, producing $(1, a^k, 1)$}
  \For{$i \leftarrow \ell'$ \KwDownTo $1$}{
    \lIf{$\kappa'_i = 1$}{$e \leftarrow d^{-1}e$}
    $(c, d) \leftarrow (c^{-1}d, c)$\;
  }
  \KwRet{$d$}\;
}
\caption{Algorithm $\textsc{BasicTargetedFibExp}$ computes a targeted exponentiation by combining two low-to-high Fibonacci exponentiations and intermediate ``rewinding'' and ''fast-forwarding'' loops.}
\label{algorithm-BasicTargetedFibExp}
\end{algorithm}

It is possible to do better by choosing different intermediate values to match.  Two alternate approaches described next offer different tradeoffs between the number of registers and the number of iterations.  

\emph{Remark.}  If $k$ is relatively prime to $r$ so that the mapping $a \mapsto a^k$ is invertible, and the element $a$ is not the group identity, then the various intermediate elements will also not be the group identity.  As a result, the group inverses, i.e., $a^{-1}$ and $b^{-1}$ in the rewinding operations, will be well defined.  Computing the inverses of these elements do not necessarily add significant complexity, depending on the group and the implementation.

\section{Dual approach}
\label{section-dual}

The intermediate values in the basic approach include both $a$ and $b$, as well as the fixed value $1$.  However, it is not necessary that either the input or the output be among the intermediate values.  One way to reduce the number of iterations is to match the outputs of the core Fibonacci exponentiation algorithms directly, rather than rewinding to a common value.  

The core algorithms compute, at the very least, the Fibonacci chain values $a^{F_{\ell}}, a^{F_{\ell+1}}$ and $b^{F_{\ell'}}, b^{F_{\ell'+1}}$, respectively, where $\ell$ and $\ell'$ are the respective number of iterations.  If instead of computing $a^k$ and $b^{\textit{kInv}}$, each core algorithm were to compute the other algorithm's Fibonacci chain values, then the rewinding and fast-forwarding loops would no longer be needed.

The core of this approach is Algorithm $\textsc{FibExpDual}$, shown in Figure \ref{algorithm-FibExpDual}, which computes two exponentials via low-to-high Fibonacci exponentiation.

\begin{algorithm}
\Indp
\SetKwInOut{Input}{Input}
\SetKwInOut{Output}{Output}
\underline{Algorithm $\textsc{FibExpDual}$} \\
\BlankLine
\Input{A group element $a$ and two non-negative integers $s$ and $t$}
\Output{$(a^s, a^t)$}
\BlankLine
\Begin{
  Let $\vec{\sigma}$ and $\vec{\tau}$ be Fibonacci representations of $s$ and $t$, respectively, padded if necessary with most significant $0$ bits to be the same length.  Let $\ell = \|\vec{\sigma}\| = \|\vec{\tau}\|$\;
  $(c, d, e, f) \leftarrow (1, a, 1, 1)$\;
  \For{$i \leftarrow 1$ \KwTo $\ell$}{
    $(c, d) \leftarrow (d, cd)$\;
    \lIf{$\sigma_i = 1$}{$e \leftarrow de$}
    \lIf{$\tau_i = 1$}{$f \leftarrow df$}
  }
  \KwRet{$(e, f)$}\;
}
\caption{Algorithm $\textsc{FibExpDual}$ computes two exponentials in parallel via low-to-high Fibonacci exponentiation.}
\label{algorithm-FibExpDual}
\end{algorithm}

Building a targeted exponentiation algorithm with this approach involves two main steps:
\begin{enumerate}
\item Run $\textsc{FibExpDual}(a, s, t)$ forward, where $s = k F_h \bmod{r}$ and $t = k F_{h+1} \bmod{r}$.
\item Run $\textsc{FibExpDual}(b, s', t')$ in reverse, where $s' = \textit{kInv} \cdot F_h \bmod{r}$, and $t' = \textit{kInv} \cdot F_h \bmod{r}$.
\end{enumerate}

Algorithm $\textsc{DualTargetedFibExp}$ (see Figure \ref{algorithm-DualTargetedFibExp}) illustrates the approach.  The intermediate values again the same up to a swap operation:  $(a^{F_h}, a^{F_{h+1}}, b^{F_h}, b^{F_{h+1}})$.  
Not counting temporary values, Algorithm $\textsc{DualTargetedFibExp}$ requires four working registers.  Its profile is thus $4 \times 2h$ --- half the number of iterations as the basic approach, with one additional register.

\begin{algorithm}
\Indp
\SetKwInOut{Input}{Input}
\SetKwInOut{Output}{Output}
\SetKw{KwDownTo}{down to}
\underline{Algorithm $\textsc{DualTargetedFibExp}$} \\
\BlankLine
\Input{A group element $a$ in a group of order $r$, and integer $k$ between $1$ and $r-1$ that is relatively prime to $r$}
\Output{$a^k$}
\BlankLine
\Begin{
  Let $h$ be a positive integer such that $r < \phi^h$\;
  $\textit{kInv} \leftarrow k^{-1} \pmod{r}$\;
  $s \leftarrow kF_h \pmod{r}$\;
  $t \leftarrow kF_{h+1} \pmod{r}$\;
  $s' \leftarrow \textit{kInv} \cdot F_h \pmod{r}$\;
  $t' \leftarrow \textit{kInv} \cdot F_{h+1} \pmod{r}$\;
   Let $\vec{\sigma}$ and $\vec{\tau}$ be Fibonacci representations of $s$ and $t$ of length $h$, respectively, padded if necessary with most significant $0$ bits\;
   Let $\vec{\sigma'}$ and $\vec{\tau'}$ be Fibonacci representations of $s'$ and $t'$ of length $h$, respectively, padded if necessary with most significant $0$ bits\;
\BlankLine
  \tcp{Run $\textsc{FibDualExp}(a, s, t)$ forward, producing $(a^{F_h}, a^{F_{h+1}}, a^{kF_h}, a^{kF_{h+1}})$}
  $(c, d, e, f) \leftarrow (1, a, 1, 1)$\;
  \For{$i \leftarrow 1$ \KwTo $h$}{
    $(c, d) \leftarrow (d, cd)$\;
    \lIf{$\sigma_i = 1$}{$e \leftarrow de$}
    \lIf{$\tau_i = 1$}{$f \leftarrow df$}
  }
  \BlankLine
  \tcp{Swap first, second pair of registers, producing $(a^{kF_h}, a^{kF_{h+1}}, a^{F_h}, a^{F_{h+1}})$}
  $(c, d, e, f) \leftarrow (e, f, c, d)$\;
  \BlankLine
  \tcp{Run $\textsc{FibDualExp}(b, s', t')$ in reverse, producing $(1, a^k, 1, 1)$}
  \For{$i \leftarrow h$ \KwDownTo $1$}{
    \lIf{$\tau'_i = 1$}{$f \leftarrow d^{-1}f$}
    \lIf{$\sigma'_i = 1$}{$e \leftarrow d^{-1}e$}
    $(c, d) \leftarrow (c^{-1}d, c)$\;
  }
  \KwRet{$d$}\;
}
\caption{Algorithm $\textsc{DualTargetedFibExp}$ computes a targeted exponentiation by combining two dual low-to-high Fibonacci exponentiations.}
\label{algorithm-DualTargetedFibExp}
\end{algorithm}

\emph{Remark.} Whereas Algorithm $\textsc{FibExpDual}$ only requires that the lengths of each pair of Fibonacci representations
be the same, Algorithm $\textsc{DualTargetedFibExp}$ ties all four together at $h$, for a practical reason:  the values of each pair of exponents depend on the lengths of the \emph{other} pair of representations.  Setting all four lengths to $h$ avoids a circular dependency.  

\section{Hybrid approach}
\label{section-hybrid}

The previous two approaches are both based on low-to-high Fibonacci algorithm, and generate Fibonacci chains which then must be either rewound or matched.  A different approach is possible by employing a \emph{high-to-low} Fibonacci algorithm, a counterpart of the traditional ``left-to-right'' binary algorithm, where bits are scanned in in \emph{decreasing} order of significance and the base is instead selectively multiplied \emph{into} the sequence of powers.  

Algorithm $\textsc{HGPExp}$, shown in Figure \ref{algorithm-HGPExp}, illustrates this approach.  

As the name suggests, Algorithm $\textsc{HGPExp}$ not only computes $a^v$, but also $a^{G(v)}$ --- a \emph{Hofstadter G pair exponentiation}.  (A proof of correctness of Algorithm $\textsc{HGPExp}$ is given in Appendix \ref{section-HGPExp-proof}.)

\begin{algorithm}
\Indp
\SetKwInOut{Input}{Input}
\SetKwInOut{Output}{Output}
\SetKw{KwDownTo}{down to}
\underline{Algorithm $\textsc{HGPExp}$} \\
\BlankLine
\Input{A group element $a$ and a non-negative integer $v$}
\Output{$(a^{G(v)},a^v)$}
\BlankLine
\Begin{
  Let $\vec{\beta}$ be a Fibonacci representation of $v$ and let $\ell = \|\vec{\beta}\|$ \;
  $(c, d, e) \leftarrow (a, 1, 1)$\;
  \For{$i \leftarrow \ell$ \KwDownTo $1$}{
    \lIf{$\beta_i = 1$}{$d \leftarrow cd$}
    $(d, e) \leftarrow (e, de)$\;
  }
  \KwRet{$(d, e)$}
}
\caption{Algorithm $\textsc{HGPExp}$ computes a high-to-low Fibonacci or Hofstadter G pair exponentiation.}
\label{algorithm-HGPExp}
\end{algorithm}

The output of Algorithm $\textsc{HGPExp}$, run in the forward direction, can be arranged to match the output of Algorithm $\textsc{FibExp}$ run in reverse, based on the following observation.

\begin{lemma}
If $\ell'$ is odd and $0 < k < \phi^{\ell'}$, then $(kF_{\ell'}, kF_{\ell'+1})$ is a Hofstadter G pair.
\label{lemma-hybrid-length}
\end{lemma}

\begin{proof}
Expand the definition of $G$ and apply the property $F_{\ell'+1} = \phi F_{\ell'} + (-\phi)^{-{\ell'}}$:
\begin{eqnarray*}
G(kF_{\ell'+1}) & = & \lfloor \phi^{-1}(kF_{\ell'+1} + 1) \rfloor \\
& = & \lfloor \phi^{-1}[k(\phi F_{\ell'} + (-\phi)^{-{\ell'}}] + 1) \rfloor \\
& = & kF_{\ell'} + \lfloor -(-\phi)^{-({\ell'}+1)}k + \phi^{-1} \rfloor \quad .
\end{eqnarray*}
Because $-(-\phi)^{-({\ell'}+1)} = -\phi^{-({\ell'}+1)}$ if $\ell'$ is odd, the result equals $kF_i$ provided that $\phi^{-(\ell'+1)}k < \phi^{-1}$, corresponding to the bound above.
\end{proof}

Algorithm $\textsc{HGPExp}(a, kF_{\ell'+1})$ thus produces $(a^{kF_{\ell'}}, a^{kF_{\ell'+1}})$.

A targeted exponentiation can therefore be computed by a hybrid of the high-to-low and low-to-high Fibonacci exponentiation algorithms:
\begin{enumerate}
\item Run $\textsc{HGPExp}(a, kF_{\ell'+1})$ forward, where $\ell'$ is an odd integer such that $0 < k < \phi^{\ell'}$ and $\textit{kInv}$ has a Fibonacci representation of length at most $\ell'$.
\item Run $\textsc{FibExp}(b, $\textit{kInv}$)$ in reverse.
\end{enumerate}

Algorithm $\textsc{HybridTargetedFibExp}$ (see Figure \ref{algorithm-HybridTargetedFibExp}) illustrates the approach.  The intermediate values, as already discussed, are $(a^{kF_{\ell'}}, a^{kF_{{\ell'}+1}}, a)$.  The number of iterations of Algorithm $\textsc{HGPExp}$ equals the length of the Fibonacci representation of $v = kF_{{\ell'}+1}$.  Given that both $k$ and $\textit{kInv}$ are between $0$ and $r-1$, it follows that $\ell' \le h$ if $h$ is odd, and $\ell' \le h+1$ if $h$ is even, so $v < rF_{h+2}$.  Applying the bound on $rF_{h+2}$ developed in the proof of Theorem \ref{theorem-MHG}, it follows that $v < F_{2h+2} - 2$.  Thus there exists at least one Fibonacci representation $\vec{\beta}$ of $v$ with length at most $2h-2$.  The Fibonacci representation of $\textit{kInv}$, including padding, meanwhile, has length at most $h$ if $h$ is odd and at most $h+1$ if $h$ is even.  The profile of Algorithm $\textsc{HybridTargetedFibExp}$ is thus $3 \times 3h$ --- roughly three quarters the number of iterations as the basic approach, and the same number of working registers.

\begin{algorithm}
\Indp
\SetKwInOut{Input}{Input}
\SetKwInOut{Output}{Output}
\SetKw{KwDownTo}{down to}
\underline{Algorithm $\textsc{HybridTargetedFibExp}$} \\
\BlankLine
\Input{A group element $a$ in a group of order $r$, and an integer $k$ between $1$ and $r-1$ that is relatively prime to $r$}
\Output{$a^k$}
\BlankLine
\Begin{
  $\textit{kInv} = k^{-1} \pmod{r}$\;
  Let $\ell'$ be an odd integer such that $r < \phi^{\ell'}$ and $\textit{kInv}$ has a Fibonacci representation of length at most $\ell'$\;
  $v \leftarrow kF_{{\ell'}+1}$\;
  \BlankLine
  \tcp{Run $\textsc{HGPExp}(a, v)$, producing $(a, a^{G(v)}, a^v)$}
  Let $\vec{\beta}$ be a Fibonacci representation of $v$\;
  $(c, d, e) \leftarrow (a, 1, 1)$\;
  \For{$i \leftarrow \|\vec{\beta}\|$ \KwDownTo $1$}{
    \lIf{$\beta_i = 1$}{$d \leftarrow cd$}
    $(d, e) \leftarrow (e, de)$\;
  }
  \BlankLine
  \tcp{Reorder registers, producing $(a^{kF_{\ell'}}, a^{kF_{{\ell'}+1}}, a)$}
  $(c, d, e) \leftarrow (d, e, c)$\;
  \BlankLine
  \tcp{Run $\textsc{FibExp}(b, \textit{kInv})$ in reverse, producing $(1, a^k, 1)$}
  Let $\vec{\kappa'}$ be a Fibonacci representation of $\textit{kInv}$, padded if necessary with most significant $0$ bits to length ${\ell'}$\;
  \For{$i \leftarrow \ell'$ \KwDownTo $1$}{
    \lIf{$\kappa'_i = 1$}{$e \leftarrow d^{-1}e$}
    $(c, d) \leftarrow (c^{-1}d, c)$\;
  }
  \KwRet{$d$}\;
}
\caption{Algorithm $\textsc{HybridTargetedFibExp}$ computes a targeted exponentiation by combining a Hofstadter G pair exponentiation and a low-to-high Fibonacci exponentiation.}
\label{algorithm-HybridTargetedFibExp}
\end{algorithm}

The following example demonstrates the hybrid approach.  Let $r = 177$ and let $k = 76$.  Then $k < \phi^{11}$ and $\textit{kInv} = k^{-1} \bmod{r} = 106$ has a Fibonacci representation of length at most $11$, so one may choose $\ell' = 11$.

To compute $a^{76}$, Algorithm $\textsc{HybridTargetedFibExp}$ first sets $v = kF_{12} = 76 \cdot 144 = 10944$.  Algorithm $\textsc{HybridTargetedFibExp}$ then computes $\textsc{HGPExp}(a, 10944)$.  After reordering the registers, the algorithm computes the reverse of $\textsc{FibExp}(b, 106)$.  

The traces of these algorithms, in terms of the exponents corresponding to the registers, are shown in Figure \ref{figure-hybrid-trace}.  For both algorithms, a Zeckendorf representation of the exponent is employed, though any Fibonacci representation would work for the first algorithm, and any Fibonacci representation with length $11$, including padding, would work for the second.  The exponents after the first algorithm are $(1, G(v), v) = (1, 6764, 10944)$.   The exponents after the swap and the second algorithm are $(0, k, 1 - \textit{kInv} \cdot k) = (0, 76, -8055)$, or equivalently $(0, 76, 0)$ modulo $179$.  The final values of the registers are thus $(a^0, a^{76}, a^0) = (1, b, 1)$, the desired targeted result.

\begin{figure}
\begin{center}
\begin{tabular}{||c||c|c|c|c||}
\hline \hline
\multicolumn{5}{||c||}{$\textsc{HGPExp}(a, 10944)$ trace} \\
\hline \hline
$i$ & $\beta_i$ & $c$ exp. & $d$ exp. & $e$ exp. \\ \hline
--- & --- & 1 & 0 & 0 \\ \hline
19 & 1 & 1 & 1 & 1 \\ \hline
18 & 0 & 1 & 1 & 2 \\ \hline
17 & 1 & 1 & 3 & 4 \\ \hline
16 & 0 & 1 & 4 & 7 \\ \hline
15 & 1 & 1 & 8 & 12 \\ \hline
14 & 0 & 1 & 12 & 20 \\ \hline
13 & 1 & 1 & 21 & 33 \\ \hline
12 & 0 & 1 & 33 & 54 \\ \hline
11 & 1 & 1 & 55 & 88 \\ \hline
10 & 0 & 1 & 88 & 143 \\ \hline
9 & 1 & 1 & 144 & 232 \\ \hline
8 & 0 & 1 & 232 & 376 \\ \hline
7 & 1 & 1 & 377 & 609 \\ \hline
6 & 0 & 1 & 609 & 986 \\ \hline
5 & 1 & 1 & 987 & 1596 \\ \hline
4 & 0 & 1 & 1596 & 2583 \\ \hline
3 & 1 & 1 & 2584 & 4180 \\ \hline
2 & 0 & 1 & 4180 & 6764 \\ \hline
1 & 0 & 1 & 6764 & 10944 \\ 
\hline \hline
\end{tabular}
\quad \quad
\begin{tabular}{||c||c|c|c|c||}
\hline \hline
\multicolumn{5}{||c||}{$\textsc{FibExp}(b, 106)$ reverse trace} \\
\hline \hline
$i$ & $\kappa'_i$ & $c$ exp. & $d$ exp. & $e$ exp. \\ \hline
--- & --- & 6764 & 10944 & 1 \\ \hline
11 & 0 & 4180 & 6764 & 1 \\ \hline
10 & 1 & 2584 & 4180 & -6763 \\ \hline
9 & 0 & 1596 & 2584 & -6763 \\ \hline
8 & 0 & 988 & 1596 & -6763 \\ \hline
7 & 0 & 608 & 988 & -6763 \\ \hline
6 & 1 & 380 & 608 & -7751 \\ \hline
5 & 0 & 228 & 380 & -7751 \\ \hline
4 & 0 & 152 & 228 & -7751 \\ \hline
3 & 1 & 76 & 152 & -7979 \\ \hline
2 & 0 & 76 & 76 & -7979 \\ \hline
1 & 1 & 0 & 76 & -8055 \\
\hline \hline
\end{tabular}
\end{center}
\caption{Example traces of Algorithms $\textsc{HGPExp}$ and $\textsc{FibExp}$ (in reverse) in the hybrid approach.}
\label{figure-hybrid-trace}
\end{figure}

\emph{Remark.}  Algorithm $\textsc{HybridTargetedFibExp}$ chooses $\ell'$ to be an odd number, so that Lemma \ref{lemma-hybrid-length} can be applied.  A corollary of Lemma \ref{lemma-hybrid-length} shows that if $\ell'$ is even and $0 < k < \phi^{\ell'+1}$, then $(kF_{\ell'}-1, kF_{{\ell'}+1}-1)$ is a Hofstadter G pair.   The intermediate value $(a^{kF_{\ell'}}, a^{kF_{{\ell'}+1}}, a)$ can thus alternatively be computed by choosing an even $\ell'$ such that $\textit{kInv}$ has a Fibonacci representation of length at most $\ell'$, computing $\textsc{HGPExp}(a, kF_{\ell'+1}-1)$, and multiplying each of the resulting values $(a^{kF_{\ell'}-1}, a^{kF_{{\ell'}+1}-1})$ by $a$.  The alternative potentially decreases the number of iterations of both Algorithms $\textsc{HGPExp}$ and $\textsc{FibExp}$, because $\ell'$ is smaller, while adding the two multiplications by $a$.

Another alternative that also potentially decreases the number of operations is to find a Hofstadter G pair $(u, v)$ that is \emph{congruent} to the intended values $(kF_{\ell'}, kF_{{\ell'}+1})$ modulo $r$, rather than matching the values exactly.  This is an example of a more general \emph{modular Hofstadter G problem}, which is to find a Hofstadter G pair $(u, v)$ such that
\begin{eqnarray*}
u & \equiv & s \pmod{r} \\
v & \equiv & t \pmod{r} \quad ,
\end{eqnarray*}
where $r$ is a positive integer and $s$ and $t$ are integers between $0$ and $r-1$.  For the purposes of the hybrid approach, the relevant parameters are $s = kF_{\ell'} \bmod{r}$ and $t = kF_{{\ell'}+1} \bmod{r}$.  One may choose $\ell' = h$ for convenience, or $\ell'$ as the possibly smaller length of the Fibonacci representation of $\textit{kInv}$.  A general solution to the modular Hofstadter G problem is presented in Appendix \ref{section-MHG}.

\section{Quantum exponentiation}
\label{section-quantum}

The design of efficient circuits for exponentiation on a quantum computer has been well studied since Shor's breakthrough quantum algorithms for integer factoring and the discrete logarithm problem in \cite{shor1999polynomial}.

In Shor's algorithm, the base $a$ for the exponentiation operation is fixed as an external input to the quantum circuit.  As a result, the successive squares of $a$, i.e., $a, a^2, a^4, \ldots$ can be precomputed and ``compiled'' into a sequence of reversible multiplication circuits.  The output $b = a^k$ can then be computed by conditionally applying the multiplication circuits, based on the corresponding bits of $k$, following the binary ``right-to-left'' exponentiation algorithm.  
Binary exponentiation in this case only requires a fixed number of quantum registers, so is well matched to Shor's algorithm, making it the focus of research in quantum exponentiation so far.

In other applications, however, the base $a$ --- or even the parameters of the group --- may be variable, generated by previous operations within a quantum circuit. For example, Bernstein \emph{et al.}'s ingenious recent combination \cite{bernstein2017post} of Grover's quantum search algorithm with the Elliptic Curve Method for integer factoring computes exponentiation operations in multiple \emph{groups} in superposition.  Though the exponent may be fixed, the parameters of the elliptic curve group and thus the successive squares in binary exponentiation will vary, and therefore can't be precompiled into the circuit.  This makes binary exponentiation less efficient because of the larger number of quantum registers required for all the successive squares, and favors Fibonacci exponentiation.  

Bernstein \emph{et al.}'s algorithm does not involve targeted exponentiation, however, because the exponential itself does not need to be carried forward into further computation.  Rather, it is sufficient just to check whether the exponential is the group identity.  This check can be made alongside any side values, even input-dependent ones; the side values as well as the exponential can then be rewound, restoring the initial input, i.e., $a$.  As a result, a core Fibonacci exponentiation algorithm such as Algorithm $\textsc{FibExp}$ or $\textsc{HGPExp}$ is sufficient, without any of the additional complexity of the targeted exponentiation approaches.

One application where targeted Fibonacci exponentiation would be relevant is when a superposition of multiple bases is raised to a fixed exponent, where the resulting superposition of exponentials is then carried forward into subsequent computation.  As further discussed in Sec. 6.3 of \cite{cryptoeprint:2017:745}, such an exponentiation could potentially transform a random eigenstate of a group operation in Shor's algorithm into a fixed eigenstate, which may be beneficial for further computation.  The targeting ensures that individual exponentials in the resulting superposition are not entangled with side values.   

\section{Conclusion}
\label{section-conclusion}

As quantum computing moves steadily toward practicality, quantum algorithm research, especially for cryptanalytic applications, has taken on a much more practical note as well, focusing on exact rather than asymptotic complexity \cite{roetteler2017quantum}.  At this point in the development of the technology, every qubit and iteration counts, making optimization particularly important.  

Although targeted exponentiation appears to have only limited applications at this time, it may nevertheless be helpful to have such algorithms among the tools to apply as new applications are explored.  The three approaches are discussed here offer tradeoffs in terms of the number of working registers and iterations, as summarized in Figure \ref{figure-comparison}.  Further improvements may be possible, perhaps not even involving Fibonacci addition chains at all.  Proving lower bounds on the number of working registers and iterations remains an open question.

\begin{figure}
\begin{center}
\begin{tabular}{||c||c|c|c||}
\hline \hline
Iterations & Basic & Dual & Hybrid \\ \hline \hline
$0$ & $(1, a, 1)$ & $(1, a, 1, 1)$ & $(a, 1, 1)$ \\
& $\updownarrow$ & $\updownarrow$ & \multirow{3}{*}{$\Bigg\updownarrow$} \\
$h$ & $(a^{F_{\ell}}, a^{F_{\ell+1}}, a^k)$ & $(a^{F_h}, a^{F_{h+1}}, a^{kF_h}, a^{kF_{h+1}})$ & \\
& $\updownarrow$ & $\updownarrow$ & \\
$2h$ & $(1, a, a^k)$ & $(1, a^k, 1, 1)$ & $(a, a^{kF_{\ell'}}, a^{kF_{\ell'+1}})$ \\ \cline{3-3}
& $\updownarrow$ & & $\updownarrow$ \\
$3h$ & $(a^{kF_{\ell'}}, a^{kF_{\ell'+1}}, a)$ & & $(1, a^k, 1)$ \\ \cline{4-4}
& $\updownarrow$ & & \\
$4h$ & $(1, a^k, 1)$ & & \\ \hline \hline
Profile & $3 \times 4h$ & $4 \times 2h$ & $3 \times 3h$ \\ \hline \hline
\end{tabular}
\end{center}
\caption{Comparison of three Fibonacci-based approaches for targeted exponentiation.  ``Profile'' is number of working registers times rough number of iterations, rounded to multiples of $h = \lceil \log_{\phi} r \rceil$.}
\label{figure-comparison}
\end{figure}

\section*{Acknowledgements}

This paper, similar to the one that motivated it \cite{cryptoeprint:2017:745}, was written on personal time, and I again thank my family for their encouragement and support.  The paper also provided an opportunity to reconnect to Prof. Peter Anderson, long-time faculty member at the Rochester Institute of Technology, who supervised my role as an adjunct faculty member at the start of my career.  I am grateful to him for his helpful technical comments as well as for his contributions, alongside many other mathematicians, to the remarkable theory and practice of Fibonacci numbers.

\bibliography{fibexp-bibliography}
\bibliographystyle{alpha}

\appendix

\section{Proof of correctness of Algorithm $\textsc{HGPExp}$}
\label{section-HGPExp-proof}

Several building blocks will help establish the correctness of Algorithm $\textsc{HGPExp}$.

\begin{lemma}
If $(u, v)$ is a Hofstadter G pair, then $(v, u + v)$ and $(v + 1, u + v + 1)$ are Hofstadter G pairs.
\label{lemma-Hofstadter-recurrence}
\end{lemma}

\begin{proof}
If $(u, v)$ is a Hofstadter G pair, then by definition, $u = G(v) = \lfloor \phi^{-1} (v + 1) \rfloor$.  Rewrite this as interval membership:
\begin{equation*}
\phi^{-1} (v + 1) \in (u : u + 1) \quad .
\end{equation*}
(An open interval on the lower bound is appropriate because $\phi^{-1}(v + 1)$ cannot be an integer.)
Add $v + 1$ to both sides:
\begin{equation*}
\phi (v + 1) \in (u + v + 1 : u + v + 2) \quad .
\end{equation*}
Now multiply by $\phi^{-1}$:
\begin{equation*}
v + 1 \in (\phi^{-1} (u + v + 1) : \phi^{-1} (u + v + 2)) \quad .
\end{equation*}
Next, rewrite as an interval for $\phi^{-1} (u + v + 1)$:
\begin{equation*}
\phi^{-1} (u + v + 1) \in (v + 1 - \phi^{-1} : v + 1) \subset (v : v + 1) \quad .
\end{equation*}
It follows that $\lfloor \phi^{-1} (u + v + 1) \rfloor = v$, which implies that $v = G(u + v)$ and that $(v, u + v)$ is a Hofstadter G pair.  The proof for $(v + 1, u + v + 1)$ is similar.
\end{proof}

\begin{lemma} If $(u, v)$ is a Hofstadter G pair and $v = \textsc{FibSum}(\vec{\beta})$ for some Fibonacci representation $\vec{\beta}$, then $u = \textsc{FibSum}(\vec{\beta} \Downarrow 1) + \beta_1$.
\label{lemma-Hofstadter-downshift}
\end{lemma}

\begin{proof}
Rewrite $G(v)$ as above and expand $v$:
\begin{equation*}
\phi^{-1} \left( \sum_{i = 1}^{\|\vec{\beta}\|} \beta_i F_{i+1} + 1 \right) \in (u : u + 1) \quad .
\end{equation*}
Now apply the property $F_{i+1} = \phi F_i + (-\phi)^{-i}$ to each term of the sum:
\begin{equation*}
\phi^{-1} \left( \sum_{i = 1}^{\|\vec{\beta}\|} \beta_i F_{i+1} + 1 \right) = 
\phi^{-1} \left( \sum_{i = 1}^{\|\vec{\beta}\|} \beta_i (\phi F_i + (-\phi)^{-i}) + 1 \right) = 
\sum_{i = 1}^{\|\vec{\beta}\|} \beta_i F_i + \epsilon + \phi^{-1} \quad ,
\end{equation*}
where the error term $\epsilon$ is defined as
\begin{equation*}
\epsilon = - \sum_{i = 1}^{\|\vec{\beta}\|} \beta_i (-\phi)^{-(i+1)} \quad .
\end{equation*}
The error term can readily be shown to bounded by $-\phi^{-1} < \epsilon < \phi^{-2}$, which implies that
\begin{equation*}
\sum_{i = 1}^{\|\vec{\beta}\|} \beta_i F_i  \in (u - \epsilon - \phi^{-1} : u + 1 - \epsilon - \phi^{-1}) \subset (u - 1 : u + 1) \quad .
\end{equation*}
The left-hand side is equivalent to $\textsc{FibSum}(\vec{\beta} \Downarrow 1) + \beta_1$.  Because the only integer in the interval on the right-hand side is $u$, the result follows.
\end{proof}

It is noteworthy that even though $v$ may have more than one Fibonacci representation $\vec{\beta}$, all lead to the same $u = G(v)$.

\emph{Remark.}  The preceding lemmas are both proved in a different way by Letuozey, without reference to the interval bounds on $v$ (\cite{letouzey2015hofstadter}, Theorems 9 and 14). 

\begin{lemma} 
If $(u, v)$ is a Hofstadter G pair and $v = \textsc{FibSum}(\vec{\beta})$, then $u + v = \textsc{FibSum}(\langle 0 \rangle \| \vec{\beta})$ and $u + v + 1 = \textsc{FibSum}(\langle 1 \rangle \| \vec{\beta})$.
\label{lemma-Hofstadter-sums}
\end{lemma}

\begin{proof}
Expand and add the Fibonacci sums for $u$ and $v$:
\begin{equation*}
u + v = \sum_{i = 1}^{\|\vec{\beta}\|} \beta_i F_i + \sum_{i = 1}^{\|\vec{\beta}\|} \beta_i F_{i+1} = \sum_{i = 1}^{\|\vec{\beta}\|} \beta_i F_{i+2} \quad .
\end{equation*}
The sum is thus equivalent to $\textsc{FibSum}(\vec{\beta} \Uparrow 1) = \textsc{FibSum}(\langle 0 \rangle \| \vec{\beta})$.  The proof for $(v + 1, u + v + 1)$ is similar.
\end{proof}

Lemmas \ref{lemma-Hofstadter-recurrence}--\ref{lemma-Hofstadter-sums} now may be applied to show the following.

\begin{theorem}
Algorithm $\textsc{HGPExp}(a, v)$ computes $(a^{G(v)}, a^v)$, i.e., $(a^u, a^v)$ where $(u, v)$ is the Hofstadter G pair associated with $v$.
\label{theorem-HGPExp}
\end{theorem}

\begin{proof}
Let $(b_{\ell+1}, c_{\ell+1}) = (1, 1)$ be the initial values of $b$ and $c$, and let $(b_i, c_i)$ be the values after the $i$th iteration of the \textbf{for} loop, where $i$ runs from $\ell$ down to $1$.  Let $(u_{\ell+1}, v_{\ell+1}) = (0, 0)$ and let $(u_i, v_i)$ denote the exponents corresponding to $(b_i, c_i)$:
\begin{eqnarray*}
u_i & = & v_{i+1} + \beta_i \\
v_i & = & u_{i+1} + v_{i+1} + \beta_i \quad .
\end{eqnarray*}

Applying Lemmas \ref{lemma-Hofstadter-recurrence}--\ref{lemma-Hofstadter-sums}, it is straightforward to see by induction that each pair $(u_i, v_i)$ is a Hofstadter G pair and that the $i$th exponent $v_i$ satisfies
\begin{equation*}
v_i = \textsc{FibSum}(\langle \beta_i \ldots \beta_{\ell} \rangle) \quad .
\end{equation*}
The final exponent pair is thus $(G(v), v) = (u, v)$, and the result follows.
\end{proof}

Given that $H$ is a group of order $r$, $\textsc{HGPExp}(a, v)$ equivalently computes $(a^s, a^t)$ where $s \equiv G(v) \pmod{r}$ and $t \equiv v \pmod{r}$.  Algorithm $\textsc{HGPExp}$ can thus be targeted to produce a specific output pair $(a^s, a^t)$ given an appropriate Hofstadter G pair, hence the motivation for the problem described next.

\section{Modular Hofstadter G problem}
\label{section-MHG}

Let $r$ be a positive integer and let $s$ and $t$ be integers between $0$ and $r-1$.  The \emph{modular Hofstadter G problem}, denoted $\textrm{MHG}(r, s, t)$, is to find a Hofstadter G pair $(u, v)$ such that
\begin{eqnarray*}
u & \equiv & s \pmod{r} \\
v & \equiv & t \pmod{r} \quad .
\label{equation-Hofstadter-G}
\end{eqnarray*}

The MHG problem arises in connection with the hybrid approach in Section \ref{section-hybrid}, but may have more general applications.  The problem may be viewed as a modular arithmetic counterpart to Anderson's algorithm for finding a pair of adjacent integers in the extended Fibonacci Zeckendorf array \cite{anderson2014extended} (see also Appendix \ref{section-anderson}), though focusing here only on the right half of the array.

The following building block will prove helpful in the solution.

\begin{lemma}
\label{lemma-FindW}
Let $r$ be a positive integer and let $\gamma$ be a real number, $0 \le \gamma < 1$.  If $r < \phi^h$ for some positive integer $h$, then there exists at least one integer $w$ in the range $0 \le w \le F_{h+2}-1$ such that $(\phi^{-1}w)_1 \in (\gamma : \gamma + 1/r)_1$.
\end{lemma}

\begin{proof}
Consider first the case that $h$ is even.  

Because $r < \phi^h < F_{h+2}$, the width of the interval $(\gamma : \gamma + 1/r)_1$ is greater than $1/F_{h+2}$, so it contains the value $j / F_{h+2}$ where $j = \lceil \gamma F_{h+2} \rceil$.  

Now consider the second interval $(\phi^{-1} w_L : \phi^{-1} w_H)_1$ where $w_L = j F_{h+1} \bmod{F_{h+2}}$ and $w_H = (j+1) F_{h+1} \bmod{F_{h+2}}$.  Because $\phi^{-1} = F_{h+1}/F_{h+2} - \phi^{-(h+2)}/F_{h+2}$ and $(F_{h+1})^2 \equiv 1 \pmod{F_{h+2}}$ for even $h$, the interval can be rewritten as
\begin{equation*}
\left (\frac{j}{F_{h+2}} - w_L \frac{\phi^{-(h+2)}}{F_{h+2}} : \frac{j+1}{F_{h+2}} - w_H \frac{\phi^{-(h+2)}}{F_{h+2}} \right )_1 \quad .
\end{equation*}

Because $\phi^{-(h+2)} < 1 / F_{h+2}$ and both $w_L$ and $w_H$ are between $0$ and $F_{h+2} - 1$, the ``error terms'' in the second interval bounds are both less than $1 / F_{h+2}$.  The second interval thus also contains $j / F_{h+2}$.  

The width of the second interval is $1/F_{h+2} + (w_L - w_H) \phi^{-(h+2)}/F_{h+2}$.  Based on the definition of $w_L$ and $w_H$, the difference $(w_L - w_H)$ is either $-F_{h+1}$ or $F_{h+2} - F_{h+1} = F_h$, so the maximum width of the interval is $1/F_{h+2} + F_h \phi^{-(h+2)}/F_{h+2}$.  Because $F_{h+2} = \phi^2 F_h + \phi^{-h}$ for even $h$, the width simplifies to
\begin{equation*}
\frac{1}{F_{h+2}} + F_h \frac{\phi^{-(h+2)}}{F_{h+2}} = \frac{1 + \phi^{-(h+2)} F_h}{\phi^2 F_h + \phi^{-h}} = \phi^{-h} \quad .
\end{equation*}
(The minimum width is $1/F_{h+2} - F_{h+1} \phi^{-(h+2)}/F_{h+2} = \phi^{-(h+1)}$.)

Because the two intervals overlap, the wider interval must contain at least one of the endpoints of the narrower.  Because $1/r > \phi^{-h}$, $(\gamma : \gamma + 1/r)_1$ is wider, and thus contains at least one of $(\phi^{-1} w_L)_1$ and $(\phi^{-1} w_H)_1$, thereby producing a solution $w$.

The proof for the case that $h$ is odd is similar, with the candidates instead defined as $w_L = j F_h \bmod{F_{h+2}}$ and $w_H = (j-1) F_h \bmod{F_{h+2}}$.
\end{proof}

\begin{lemma}
Algorithm $\textsc{FindW}$, shown in Figure \ref{algorithm-FindW},  finds a multiple of $\phi^{-1}$ in the interval $(\gamma : \gamma + 1/r)_1$.  
\end{lemma}

\begin{proof}
This follows from the steps in Lemma \ref{lemma-FindW} for computing $j$, $w_H$, and $w_L$.
\end{proof}

\emph{Remark.}  Algorithm $\textsc{FindW}$ may be simplified by computing $w_H$ only if needed, and by observing that $w_H = w_L + F_{h+1} \bmod F_{h+2}$ regardless whether $h$ is even or odd.

\begin{algorithm}
\Indp
\SetKwInOut{Input}{Input}
\SetKwInOut{Output}{Output}
\underline{Algorithm $\textsc{FindW}$} \\
\BlankLine
\Input{A positive integer $r$ and a real number $\gamma$ between $0$ and $1$}
\Output{A non-negative integer $w$ such that $(\phi^{-1}w)_1 \in (\gamma : \gamma + 1/r)_1$}
\BlankLine
\Begin{
  Let $h$ be the least positive integer such that $r < \phi^h$. \\
  $j \leftarrow \lceil \gamma F_{h+2} \rceil$\;
  \uIf{$h$ is even}{
    $w_L \leftarrow j F_{h+1} \bmod F_{h+2}$\;
    $w_H \leftarrow (j + 1) F_{h+1} \bmod F_{h+2}$\;
  }
  \Else{
    $w_L \leftarrow j F_h \bmod F_{h+2}$\;
    $w_H \leftarrow (j - 1) F_h \bmod F_{h+2}$\;
  }
  \lIf{$(\phi^{-1} w_L)_1 \in (\gamma : \gamma + 1/r)_1$}{\KwRet{$w_L$}}\;
  \lElse{\KwRet{$w_H$}}\;
}
\caption{Algorithm $\textsc{FindW}$ finds a multiple of $\phi^{-1}$ in a specified interval modulo $1$.}
\label{algorithm-FindW}
\end{algorithm}

\begin{theorem}
If $r < \phi^h$, then there exists at least one solution $(u, v)$ to $\textrm{MHG}(r, s, t)$ such that $v \le F_{2h+2} - 2$.
\label{theorem-MHG}
\end{theorem}

\begin{proof}
It is sufficient to find a nonnegative integer $w$ such that
\begin{equation*}
G(w r + t) = \lfloor \phi^{-1} (wr + t + 1) \rfloor \equiv s \pmod{r} \quad .
\end{equation*}
Rewrite the equation as interval membership modulo $r$ (observing as previously that the integer bound $s$ is not achievable):
\begin{equation*}
(\phi^{-1} (wr + t + 1))_r \in (s : s + 1)_r \quad .
\end{equation*}
Now move all terms except $\phi^{-1}wr$ to the right and normalize mod $1$:
\begin{equation*}
(\phi^{-1} w)_1 \in (\gamma : \gamma + 1/r)_1 \quad ,
\end{equation*}
where $\gamma = ((s - \phi^{-1}(t + 1))/r)_1$.  

By Lemma \ref{lemma-FindW}, there exists at least one such integer $w$ in the range $0 \le w \le F_{h+2}-1$ such that $(\phi^{-1}w)_1 \in (\gamma : \gamma + 1/r)_1$.  Fix such a $w$ and let $v = wr + t$.  If follows that $(G(v), v)$ is a solution to the MHG problem.

Next, consider the bound on $v$.  Based on the bounds on $t$ and $w$, it it easy to see that $0 \le v \le r F_{h+2} - 1$.  The maximum integer value of $r$, $\lfloor \phi^h \rfloor$, is $\phi^h + \phi^{-h} - 1$ if $h$ is even and $\phi^h - \phi^{-h}$ if $h$ is odd, following observations by Caveney and Catalini on OEIS sequence A014217.  For even $h$, this gives the following bound on $r F_{h+2}$:
\begin{eqnarray*}
r F_{h+2} & \le & \frac{\phi^{h+2} - \phi^{-(h+2)}}{\sqrt{5}} \cdot (\phi^h + \phi^{-h} - 1) \\
& = & \frac{\phi^{2h+2} - \phi^{h+2} + \phi^2 - \phi^{-2} + \phi^{-(h+2)} - \phi^{-(2h+2)}}{\sqrt{5}} \\
& = & F_{2h+2} - F_{h+2} + F_2 \quad ,
\end{eqnarray*}
which is at most $F_{2h+2} - 2$.  For odd $h$, the bound is
\begin{eqnarray*}
r F_{h+2} & \le & \frac{\phi^{h+2} + \phi^{-(h+2)}}{\sqrt{5}} \cdot (\phi^h - \phi^{-h}) \\
& = & \frac{\phi^{2h+2} - \phi^2 + \phi^{-2} - \phi^{-(2h+2)}}{\sqrt{5}} \\
& = & F_{2h+2} - 1 \quad .
\end{eqnarray*}
Thus, for both even and odd $h$, it follows that $0 \le v \le r F_{h+2} - 1 \le F_{2h+2} - 2$.  
\end{proof}

\begin{corollary}
There exists at least one solution $(u, v)$ such that the Fibonacci representation of $v$ is at most $2h-1$ bits long.
\end{corollary}

\begin{corollary}
There exists at least one solution $(u, v)$ such that the Zeckendorf representation of $v$ is at most $2h$ bits long.
\end{corollary}

\begin{theorem}
Algorithm $\textsc{SolveMHG}$, shown in Figure \ref{algorithm-SolveMHG}, solves the MHG problem.  
\end{theorem}

\begin{proof}
This follows from the steps in Theorem \ref{theorem-MHG} for computing $\gamma$, $w$, $v$ and $u$.
\end{proof}

\emph{Remark.}  Algorithms $\textsc{SolveMHG}$ and $\textsc{FindW}$ may be ``rationalized'' by replacing $\phi^{-1}$ with the approximation $F_{h+1}/F_{h+2}$.  $\textsc{FindW}$ would then begin by setting $j$ as
\begin{equation*}
j \leftarrow \left \lceil \frac{s F_{h+2} - (t+1) F_{h+1}}{r} \right \rceil \bmod{F_{h+2}} \quad .
\end{equation*}
The proofs can be modified to accommodate such approximations, supporting a modified algorithm that involves only integer operations.

\begin{algorithm}
\Indp
\SetKwInOut{Input}{Input}
\SetKwInOut{Output}{Output}
\underline{Algorithm $\textsc{SolveMHG}$} \\
\BlankLine
\Input{A positive integer $r$ and two integers $s$ and $t$ between $0$ and $r-1$}
\Output{Two non-negative integers $u$ and $v$ such that $u \equiv s \bmod{r}$, $v \equiv t \bmod{r}$, and $u = G(v)$}
\BlankLine
\Begin{
  $\gamma \leftarrow (s/r - \phi^{-1}(t + 1)/r)_1$\;
  $w \leftarrow \textsc{FindW}(\gamma, r)$\;
  $v \leftarrow w r + t$\;
  $u \leftarrow G(v)$\;
  \KwRet{$(u, v)$}\;
}
\caption{Algorithm $\textsc{SolveMHG}$ solves the modular Hofstadter G problem.}
\label{algorithm-SolveMHG}
\end{algorithm}

The following example illustrates Algorithms $\textsc{FindW}$ and $\textsc{SolveMHG}$.

Let $r = 177$, $s = 141$, and $t = 25$.  Because $r < \phi^{11}$, one may choose $h = 11$.

With these parameters, the interval $(\gamma : \gamma + 1/r)_1$ is bounded by
\begin{eqnarray*}
\gamma & = & \left(\frac{s - \phi^{-1}(t+1)}{r}\right)_1 = \left(\frac{141 - \phi^{-1}\cdot 26}{179}\right)_1 \approx 0.6979392 \\
\gamma + \frac{1}{r} & = & \left(\frac{s - \phi^{-1}(t+1) + 1}{r}\right)_1 = \left(\frac{141 - \phi^{-1}\cdot 26 + 1}{179}\right)_1 \approx 0.7035258 \quad .
\end{eqnarray*}

Algorithm $\textsc{FindW}$ computes
\begin{eqnarray*}
j & = & \lceil \gamma F_{h+2} \rceil \approx \lceil 0.6979392 \cdot 233 \rceil = 163 \\
w_L & = & j F_h \bmod{F_{h+2}} = 163 \cdot 89 \bmod{233} = 61 \quad .
\end{eqnarray*}
Because $(\phi^{-1} \cdot 61)_1 \approx 0.7000734$ is in the interval, the algorithm returns $w = 61$.

Algorithm $\textsc{SolveMHG}$ computes 
\begin{equation*}
v = wr + t = 61 \cdot 179 + 25 = 10944 \quad .
\end{equation*}
The corresponding value of $u$ is
\begin{equation*}
u = G(v) = \lfloor \phi^{-1} (v + 1) \rfloor = \lfloor \phi^{-1} \cdot 10945 \rfloor = 6764 \quad .
\end{equation*}
The correctness of the solution is confirmed by the congruences
\begin{eqnarray*}
10944 & \equiv & 25 \pmod{179} \\
6764 & \equiv & 141 \pmod{179} \quad .
\end{eqnarray*}

These parameters match the ones in Section \ref{section-hybrid} and Figure \ref{figure-hybrid-trace}, showing that the solution to the MHG problem, in this particular case with $h$ odd, gives rise to the same exponent $v$ as would be directly obtained by computing $v = kF_{h+1}$.

\section{Revisiting Anderson's algorithm}
\label{section-anderson}

Anderson \cite{anderson2014extended} recently showed how to solve a problem related to the MHG problem:  to locate a given pair of adjacent integers in the extended Fibonacci Zeckendorf array.  Morrison \cite{morrison1980stolarsky} had previously proved that every pair of positive integers appears exactly once as adjacent elements in the array.  Anderson extended the result to show that every pair of integers $(u, v)$ such that $v\phi + u > 0$ appears exactly once, and also gave three algorithms for locating these pairs, which may be called \emph{Anderson pairs}.

Anderson's first algorithm for locating a pair $(u, v)$ involves computing the Fibonacci successors of $(u, v)$ until a pair is reached whose Zeckendorf representations are single-index shifts of one another.  Assuming $(u, v)$ is in the left half of the array, such a condition signals that the resulting pair has reached the right half.  Indeed, although not called out in the algorithm, the recursion can stop as soon as the \emph{second} element of the resulting pair has reached the right half.  This can be also detected by checking whether the resulting pair is a Hofstadter G pair.

The hybrid approach to targeted exponentiation in Section \ref{section-hybrid} follows a similar pattern of recursion to Anderson's first algorithm, but in reverse.  Consider again the values of the exponent of $b$ generated in the trace in Figure \ref{figure-hybrid-trace}, moving backwards from the end (and writing left to right, with extended Fibonacci Zeckendorf array column numbers at the top):
\begin{equation*}
\begin{tabular}{||c|c|c|c|c|c|c|c|c|c|c|c||}
\hline \hline
-8 & -7 & -6 & -5 & -4 & -3 & -2 & - 1 & 0 & 1 & 2 & 3 \\ \hline
76 & 76 & 152 & 228 & 380 & 608 & 988 & 1596 & 2584 & 4180 & 6764 & 10944 \\
\hline \hline
\end{tabular}
\end{equation*}

The first Hofstadter G pair in this sequence is $(2584, 4180)$, marking the transition from the left half of the array to the right half.  Anderson's first algorithm would locate $(76, 76)$ in the left half of the array by recursing it $9$ times until it reaches $(4180, 6764)$.  The transition could also be detected at $(2584, 4180)$.  However, testing of successive pairs isn't needed at all, as the following observations show.  

\begin{lemma}
Define the function
\begin{equation*}
\delta(u, v) \stackrel{\Delta}{=} -u + \phi^{-1} v \quad .
\end{equation*}
If $(u, v)$ is a Hofstadter G pair, then $\delta(u, v) \in (-\phi^{-1} : \phi^{-2})$.
\label{lemma-Hofstadter-delta}
\end{lemma}
\begin{proof}
Expand $G(v)$ as usual and rewrite it as interval membership, as in Lemma \ref{lemma-Hofstadter-recurrence}:
\begin{equation*}
\phi^{-1}(v + 1) \in (u : u + 1) \quad .
\end{equation*}
Now subtract $u + \phi^{-1}$ from both sides:
\begin{equation*}
-u + \phi^{-1} v \in (-\phi^{-1} : \phi^{-2}) \quad .
\end{equation*}
The proof follows.
\end{proof}

\begin{lemma}
If $(u, v)$ is a Hofstadter G pair where $v$ is in column $1$ of the extended Fibonacci Zeckendorf array, i.e., the Zeckendorf representation $\vec{\beta}$ of $v$ has $\beta_1 = 1$, then $\delta(u, v) \in (-\phi^{-1} : -\phi^{-3})$.
\label{lemma-Hofstadter-first-column}
\end{lemma}

\begin{proof}
The result can be shown by similar analysis to the error bound in Lemma \ref{lemma-Hofstadter-downshift}, with the additional condition that $\beta_1 = 1$ and $\beta_2 = 0$ (as required by Zeckendorf form).  The upper bound on the interval in Lemma \ref{lemma-Hofstadter-delta} is thus reduced by $\phi^{-2}$ (because the error term corresponding to $\beta_1$, i.e., $-\phi^{-2}$, is always present), and again by $\phi^{-3}$ (because the error term corresponding to $\beta_2$, i.e., $\phi^{-3}$, is not).
\end{proof}

\begin{lemma}
If $(u, v)$ is a Hofstadter G pair where $v$ is in column $j \ge 1$ of the extended Fibonacci Zeckendorf array, i.e., the Zeckendorf representation $\vec{\beta}$ of $v$ has $\beta_1 = \cdots = \beta_{j-1} = 0$ and $\beta_j = 1$, then $\delta(u, v) \in (-\phi^{-j} : -\phi^{-(j+2)})$ if $j$ is odd, and $\delta(u, v) \in (\phi^{-(j+2)} : \phi^{-j})$ if $j$ is even.
\label{lemma-Hofstadter-right-half}
\end{lemma}

\begin{proof}
The interval can be shown by induction.  Lemma \ref{lemma-Hofstadter-delta} covers $j = 1$.  Now suppose that $(u, v)$ is a Hofstadter G pair where $v$ is in column $j - 1$ for $j > 1$.  By definition of the array, $(u, u+v)$ moves one column to the right, i.e., $u+v$ is in column $j$.  The pair is also a Hofstadter G pair by Lemma \ref{lemma-Hofstadter-recurrence}.  Now consider the value $\delta(v, u+v)$:
\begin{equation*}
\delta(v, u+v) = -v + \phi^{-1}(u+v) = -\phi^{-1}u - \phi^{-2}v = -\phi^{-1}\delta(u, v) \quad .
\end{equation*}
The bounds of the interval are thus multiplied by $-\phi^{-1}$ with each move, and the result follows.
\end{proof}

These lemmas, which recurse the intervals to the right, also show that every Hofstadter G pair $(u, v)$ occurs exactly once in a ``one-column-extended'' Fibonacci Zeckendorf array that also includes column $0$ from the left hand side.  This array may be considered a \emph{semi-extended Wythoff array} by analogy with the extended Wythoff array that begins with column $-1$ per OEIS sequence A033513.

The intervals can be precursed to the left as well:
\begin{lemma}
If $(u, v)$ is an Anderson pair where $v$ is in column $j \le 0$ of the extended Fibonacci Zeckendorf array, then $\delta(u, v) \in (-\phi^{-j} : -\phi^{-(j+2)})$ if $j$ is odd, and $\delta(u, v) \in (\phi^{-(j+2)} : \phi^{-j})$ if $j$ is even.
\end{lemma}
\begin{proof}
The proof is similar to Lemma \ref{lemma-Hofstadter-right-half} with the precursion multiplying the bounds of the interval by $-\phi$:
\begin{equation*}
\delta(v-u, u) = -(v-u) + \phi^{-1}u = \phi u - v = -\phi \delta(u, v) \quad .
\end{equation*}
\end{proof}

The intervals thus alternate between positive and negative based on the column of $v$:
\begin{equation*}
\begin{tabular}{||c||c|c|c|c|c||}
\hline \hline
Column & $-2$ & $-1$ & $0$ & $1$ & $2$ \\ \hline
Interval & $(1 : \phi^2)$ & $(-\phi : -\phi^{-1})$ & $(\phi^{-2} : 1)$ & $(-\phi^{-1} : -\phi^{-3})$ & $(\phi^{-4} : \phi^{-2})$ \\ \hline
Sign & $+$ & $-$ & $+$ & $-$ & $+$ \\ \hline
$\log_{\phi}$ & $(0 : 2)$ & $(-1 : 1)$ & $(-2 : 0)$ & $ (-3 : -1)$ & $(-4 : -2)$ \\
\hline \hline
\end{tabular}
\end{equation*}
The third row indicates the sign of the bounds of the interval, and fourth row shows the range of logarithm base $\phi$ of the absolute value of the bounds.  Because the intervals are mutually exclusive and collectively exhaustive (ignoring the exact powers of $\phi$ on the interval boundaries, which cannot be achieved), the value $\delta(u, v)$ falls in exactly one interval.  This means that it is possible to determine the column of $v$ directly from $\delta(u, v)$.  

\begin{theorem}
Let $(u, v)$ be an Anderson pair.  Let $m = \lceil \log_{\phi} |\delta(u, v)| \rceil$.  Then $v$ is located at column $j$ of the extended Fibonacci Zeckendorf array, where $j$ is defined as:
\begin{equation*}
j = \left\{
\begin{array}{ll}
2 \lfloor (-m/2) \rfloor & \textnormal{if $\delta(u, v) > 0$} \\
2 \lfloor ((1-m)/2) \rfloor - 1 & \textnormal{otherwise}
\end{array}
\right .
\quad .
\end{equation*}
The pair is located in row $n = uF_{-1-j} + vF_{-j}$.
\end{theorem}

\begin{proof}
The column calculation follows the pattern in the table above.  The row number follows the movement of the pair. Precursing the pair back $j$ columns if $j > 0$, or recursing forward $-j$ columns, if $j < 0$, produces $(uF_{-1-j} + vF_{-j}, uF_{-j} + vF_{1-j})$.  At this point, the left element of the pair is in column $-1$, whose values give the row number of the extended Wythoff array, per OEIS sequence A033513.
\end{proof}

Although Anderson described his first algorithm as ``inefficient,'' the theorem shows that the algorithm actually can be made very efficient indeed.

In addition to locating an Anderson pair $(u, v)$, the theorem also provides a strategy for generalizing Algorithm $\textsc{HGPExp}$ to generate the exponentials $(a^u, a^v)$ for any such pair.  The resulting \emph{Anderson pair exponentiation algorithm} would first compute the column location $j$ of $v$.  If the column is in the right half of the array, i.e., $j \ge 1$, then the algorithm would proceed as in Algorithm $\textsc{HGPExp}$.  If it's in the left half, however, the algorithm would proceed as in Algorithm $\textsc{HGPExp}$ as far as the pair $(uF_{-j} + vF_{1-j}, uF_{1-j} + vF_{2-j})$ --- the one at the transition from the left half to the right --- and then precurse by $1-j$ columns.  In contrast to Algorithm $\textsc{FibExpDual}$, which could compute the two exponentials in fewer iterations, this new \emph{Anderson pair exponentiation algorithm} requires only three working registers, not four. 

Finally, a corollary of Lemma \ref{lemma-Hofstadter-first-column} provides an algorithm for determining the Zeckendorf representation of an integer from \emph{low to high}, as an alternative to the conventional ``greedy'' algorithm, which operates high to low.  
\begin{corollary}
If $v$ is in column $1$ of the extended Fibonacci Zeckendorf array, then $(\phi^{-1} v)_1 \in (\phi^{-2} : 1 - \phi^{-3})$, otherwise $(\phi^{-1} v)_1 \in (0 : \phi^{-2}) \cup (1 - \phi^{-3} : 1)$.
\end{corollary}
The Zeckendorf representation can thus be determined a bit at a time by checking the value of $(\phi^{-1} v)_1$, subtracting the bit, down-shifting by computing the Hofstadter G function, and repeating.  This process can also be ``rationalized'' similar to the remark on Algorithms $\textsc{SolveMHG}$ and $\textsc{FindW}$ in Appendix \ref{section-MHG} by approximating $\phi^{-1}$ as a ratio of Fibonacci numbers, with appropriate adjustments to the interval bounds.  It is also possible to enumerate the set of possible Fibonacci representations low to high by a similar approach, with overlapping intervals corresponding to available choices of $0$ and $1$ as least significant bits.  The low-to-high algorithm may be beneficial in implementations where the least significant bit is needed first, e.g., for the exponent in Algorithm $\textsc{FibExp}$ or $\textsc{FibExpDual}$ when run in the forward direction, and where it is preferable to compute the representation one bit at a time to save space, rather than all at once.

\end{document}